\documentclass[oneside]{amsart}
\usepackage[T2A]{fontenc}
\usepackage[cp1251]{inputenc}
\usepackage{amssymb}
\usepackage{amsxtra}
\usepackage{amsmath}
\usepackage{amsthm}
\usepackage[all]{xy}

\usepackage{hyperref}

\usepackage[russian,english]{babel}

\DeclareMathOperator{\GL}{GL}

\DeclareMathOperator{\Lie}{Lie}
\DeclareMathOperator{\Cent}{Cent}

\DeclareMathOperator{\Spec}{Spec}

\DeclareMathOperator{\Gm}{{\mathbf G}_m}

\DeclareMathOperator{\ZZ}{{\mathbb Z}}

\DeclareMathOperator{\et}{\text{\it \'et}}

\newtheorem{lem}{Lemma}[section]
\newtheorem*{thm*}{Theorem}
\newtheorem{thm}[lem]{Theorem}

\newtheorem{cor}[lem]{Corollary}
\theoremstyle{definition}{    }
\theoremstyle{definition}{    }
\theoremstyle{definition}{  \newtheorem{conj}[lem]{Conjecture} }



\begin{document}

\title{On a theorem of Harder}

\author{Ivan Panin}
\address{St. Petersburg Department of Steklov Mathematical Institute, nab. r. Fontanki 27, 191023, St. Petersburg, Russia}
\email{paniniv@gmail.com}

\author{Anastasia Stavrova}\address{
St. Petersburg Department of Steklov Mathematical Institute, nab. r. Fontanki 27, 191023, Saint Petersburg, Russia
}

\email{anastasia.stavrova@gmail.com}

\maketitle

\section{Introduction}

In 1967 G\"unter Harder proved the following theorem.

\begin{thm}\cite[Satz 3.3]{Har67}\label{th:harder}
Let $G$ be a simply connected semisimple group over a Dedekind domain $A$. Assume that $G$
is quasi-split over the  fraction field of $A$. Then $H^1_{Zar}(A,G)=1$.
\end{thm}

Our aim is to generalize this theorem to isotropic simply connected semisimple groups.

We use the following definition of an isotropic group.
Let $G$ be a reductive group (scheme) in the sense of~\cite{SGA3} over a commutative ring $R$.
A parabolic $R$-subgroup $P$ of $G$ is called {\it strictly proper}, if for each ring homomorphism
$s:R\to\Omega(s)$, where
$\Omega(s)$ is an algebraically closed field,
the type of the parabolic subgroup $P_{\Omega(s)}$ in the sense of~\cite[Exp. XXVI, \S 3.2]{SGA3}
does not contain any connected component of the Dynkin diagram of $G_{\Omega(s)}$.
We will say that
$G$ is {\it strictly isotropic}, if it contains a strictly proper parabolic $R$-subgroup.

Our main result is the following theorem.

\begin{thm}\label{th:main}
Let $A$ be a Dedekind domain.
Let $G$ be a simply connected
semisimple group over $A$. Assume that $G$ is strictly isotropic over the fraction field $L$ of $A$.
Then the natural map
$$
H^1_{\et}(A,G)\to H^1_{\et}(L,G)
$$
has trivial kernel.

In other words, let $E$ be a $G$-torsor over $A$. If $E$ is trivial over $L$, then $E$ is trivial.
\end{thm}

The following immediate corollary generalizes Harder's theorem.

\begin{cor}~\label{cor:Zar}
Let $A$ be a Dedekind domain.
Let $G$ be a simply connected
semisimple group over $A$. Assume that $G$ is strictly isotropic over the fraction field $L$ of $A$.
 Then $H^1_{Zar}(A,G)=1$.
\end{cor}

We note that Harder obtained Theorem~\ref{th:harder} as a corollary of the following statement: if the $L$-group
$G_L$ satisfies strong approximation, then $H^1_{Zar}(A,G)=1$~\cite[Kor. 2.3.2]{Har67}.
(In fact, this implication holds for any flat finite type $A$-group $G$, not just for simply connected semisimple groups.)
Strong approximation is known to hold for simply connected strictly isotropic groups $G$ such that
$G(\hat L_p)$ is $R$-trivial for every completion $\hat L_p$ of $L$ with respect a prime $p$ of $A$~\cite[Lemme 5.6]{Gil}.
If $L$ is a global field, a simply connected group $G$ over $L$ satisfies strong approximation under somewhat weaker
isotropy conditions, thanks to the classical results of Kneser, Platonov, and Prasad~\cite{Kne65,Pla69,Pra77}
In all these cases the
conclusion of Corollary~\ref{cor:Zar} follows from the above-cited result of Harder~\cite[Kor. 2.3.2]{Har67}.
Our proof does not use these results.

\section{Proof of the main theorem}

\begin{lem}\label{lem:spp}
Let $A$ be a Dedekind domain and let $L$ be the fraction field of $A$. Let $G$ be a reductive group over $A$. If $G$
is strictly isotropic over $L$, then $G$ is strictly isotropic over $A$.
\end{lem}
\begin{proof}
Let $\mathcal{P}$ be the smooth projective $A$-scheme of all parabolic subgroups of $G$ in the sense
of~\cite[Exp. XXVI, Cor. 3.5]{SGA3}. By~\cite[Corollaire 7.3.6]{EGAII} the natural map $\mathcal{P}(A)\to\mathcal{P}(L)$
is surjective. Since $G$ is strictly isotropic over $L$, there is a strictly proper parabolic subgroup
$P\in\mathcal{P}(L)$. Let $\tilde P\in\mathcal{P}(A)$ be a preimage of $P$. We claim that $\tilde P$ is a
strictly proper parabolic subgroup of $G$. Indeed,
since $\Spec(A)$ is connected, by~\cite[Exp. XXII, Proposition 2.8, Lemme 5.2.7]{SGA3}
it suffices to check that $\tilde P$ is strictly proper at any fixed point of $\Spec(A)$.
Thus, $\tilde P$ is strictly proper, since $\tilde P_K=P$ is strictly proper by assumption.
\end{proof}

For any commutative ring $R$, we denote by $E_P(R)$ the subgroup $G(R)$
generated by the $R$-points of the unipotent radicals of $P$ and of an opposite parabolic subgroup
$P^-$. Note that choosing a different opposite parabolic subgroup $P^-$ does not change $E_P(R)$ by~\cite[Exp. XXVI,
Cor. 1.8]{SGA3}.

\begin{proof}[Proof of Theorem~\ref{th:main}]
If $E$ is not trivial, then, since $E$ is trivial over $L$, there is a non-invertible element $f\in A$ such that $E$ is trivial over $A_f$.
Let $p_1,\ldots,p_k$ be all distinct maximal ideals of $A$ dividing the ideal $(f)$.
Let $\hat A$ be the completion of $A$ with respect to the ideal $(f)$. Then $\hat A$ is isomorphic to the direct product of rings
$\hat A_1,\ldots,\hat A_k$, where each $\hat A_i$ is the completion of $A$ with respect to the ideal $p_i$.
One has $A/f\cong\hat A/f$. Let $\hat L_i$ denote the fraction field of $\hat A_i$. Then the localization
$(\hat A)_f$ is isomorphic to the direct product $\hat L$ of the fraction fields $\hat L_i$ of the rings $\hat A_i$.

By the Serre--Grothendieck conjecture for
semilocal Dedekind rings~\cite{Guo-ded,Ni} the torsor $E$ is trivial over $\hat A_i$, since $E$ is trivial over
its fraction field $\hat L_i$ which is a field extension of $L$. Consequently, the torsor $E$ is trivial over $\hat A$.

Consider the commutative square
\begin{equation*}
\xymatrix@R=15pt@C=20pt{
A\ar[r]\ar[d]&A_{f}\ar[d]\\
\hat A\ar[r]&\hat L\\
}
\end{equation*}
This is a patching diagram for torsors. The torsor $E$ is trivial over $A_{f}$ and over $\hat A$. Then there is
an element $g\in G(\hat L)$ that patches together the two trivial $G$-torsors over $\hat L$
to get the torsor $E$.

By Lemma~\ref{lem:spp} $G$ has a strictly proper parabolic subgroup $P$ over $A$. Let $L$ be a
Levi subgroup of $P$ over $A$. Since $\Spec(A)$ is connected,
by~\cite[Th. 7.3.1]{Gi4} there is an $A$-group homomorphism
$\lambda:\Gm_{,A}\to L$ such that
$L=\Cent_G(\lambda(\Gm_{,A}))$ and $P$ is one of the two opposite parabolic $A$-subgroups $P=P^+$ and $P^-$
of $G$ uniquely determined by $\lambda$ in the following way described in~\cite[Exp. XXVI, Prop. 6.1]{SGA3}.
Let $X^*(\lambda(\Gm_{,A}))\cong\ZZ$ be the lattice
of characters of $\lambda(\Gm_{,A})\cong\Gm_{,A}$. The adjoint action of $\lambda(\Gm_{,A})$ on $\Lie(G)$
defines a $\ZZ$-grading on $\Lie(G)$, such that $\Lie(L)=\Lie(G)_0$, and $P,P^-$ are the unique parabolic subgroups of $G$ such that
$$
\Lie(P)=\Lie(L)\oplus\bigoplus\limits_{\chi\in \ZZ_{>0}}\Lie(G)_\chi\mbox{\quad and\quad
 }
\Lie(P^-)=\Lie(L)\oplus\bigoplus\limits_{\chi\in\ZZ_{<0}}\Lie(G)_\chi.
$$
This means that $P,P^-$ are a pair of opposite parabolic subgroups of $G$
that correspond to $S=\lambda(\Gm_{,A})$ in the sense of~\cite[Lemma 2.4]{PaSt}. These subgroups are strictly proper.
Then by~\cite[Lemme 4.5]{Gil} one has $G(\hat L_i)=G(\hat A_i)E_P(\hat L_i)$ for all $1\le i\le k$.

Since the base $\Spec(A)$ is affine, normal and Noetherian, by~\cite[Corollary 3.2]{Thomason} $G$ is $A$-linear,
that is, there is a closed $A$-embedding $G\to\GL_{n,A}$ for some $n\ge 1$. Clearly, we can assume that $n\ge 3$
without loss of generality. Then all the conditions of~\cite[Lemma 3.2]{PaSt} are satisfied, and it follows that
$E_P(\hat L)\subseteq G(\hat A)E_P(A_f)$. Summing up, we have
$$
g\in G(\hat L)=G(\hat A)E_P(\hat L)\subseteq G(\hat A)E_P(A_f).
$$
It follows that $E$ is trivial over $A$.

\end{proof}

\section{Nisnevich conjecture}

In 1989 Y. Nisnevich formulated the following

\begin{conj}\cite[Conjecture 1.3]{Ni89}\label{conj:nis}
Let $R$ be a regular local ring of dimension $d\ge 1$, let $m$ be the maximal ideal of $R$ and let $u\in m\setminus m^2$
be an element.
Let $G$ be a strictly isotropic reductive group over $R$. Then $H^1_{Zar}(R_u,G)=1$.
\end{conj}

Nisnevich established this conjecture for $\dim R=2$ and $G$ quasi-split~\cite[Prop. 5.1]{Ni89}.

Strictly speaking,
the initial formulation of this conjecture by Nisnevich did not include the isotropy condition
(although this condition was imposed on $G$ in the text right before and after the formulation
of~\cite[Conjecture 1.3]{Ni89}). However, R. Fedorov~\cite{Fed-nis} provided an example of an anisotropic adjoint semisimple group for which the same
statement does not hold, so we include this condition in the statement.

Fedorov also established this conjecture under the assumption that $R$ contains an infinite field. He also formulated
and proved a suitable extension of this statement to semilocal regular rings.

K. \v{C}esnavi\v{c}ius~\cite{C-bq} extended Fedorov's result to the case where
$R$ contains a finite field, and under the weaker assumption that only $G_{R/u}$ is strictly isotropic. On top of that,
 \v{C}esnavi\v{c}ius proved the following mixed characteristic statement:
$H^1_{Zar}(R_u,G)=1$ assuming that $R$ is semilocal, geometrically regular over a Dedekind subring
$D$ such that $u\in D$, and $G$ is quasi-split. Thus, the equicharacteristic case of the conjecture is completely settled.

Nisnevich conjecture is a direct generalization of the so-called Quillen's question.

\begin{conj}\cite[p. 170]{Q}\label{conj:Q}
Let $R$ be a regular local ring with maximal ideal $m$, and let $f\in m\setminus m^2$. Then
every finitely generated projective $R_f$-module is free.
\end{conj}

On top of the cases mentioned above, Conjecture~\ref{conj:Q} is known for all 3-dimensional regular local rings
$R$~\cite{Gabber-thesis}.

Our theorem yields the following two cases of Conjecture~\ref{conj:nis}.

\begin{thm}
Let $(R,m)$ be a 2-dimensional regular local ring with the fraction field $L$. Let $u\in m$ be any non-zero element.
Let $G$ be a strictly isotropic simply connected group over $R_u$. Then $H^1_{\et}(R_u,G)\to H^1_{\et}(L,G)$ has trivial kernel. As a consequence,
$H^1_{Zar}(R_u,G)=1$.
\end{thm}
\begin{proof}
The ring $R_u$ is a Dedekind domain. Thus it is subject to Theorem~\ref{th:main}.
\end{proof}

\begin{thm}\label{th:3dim}
Let $D$ be a discrete valuation ring of mixed characteristic such that its residue field is perfect. Let $(R,m)$ be a
3-dimensional local
ring, essentially smooth over $D$. Let $f\in m\setminus m^2$ be a regular element of $R$. Let $G$
be a simply connected strictly isotropic group over $D$. Then $H^1_{\et}(R_f,G)\to H^1_{\et}(L,G)$ has trivial kernel,
where $L$ is the fraction field of $R$.
\end{thm}

For the proof of this theorem we need to establish some lemmas.

\begin{lem}\label{lem:A[x]}
Let $A$ be a discrete valuation ring and let $L$ be the fraction field of $A$.
Let $G$ be a strictly isotropic reductive group over $A$.
Then $H^1_{\et}(A[x],G)\to H^1_{\et}(L[x],G)$ has trivial kernel.
\end{lem}
\begin{proof}
Since the Serre-Grothendieck conjecture holds for any reductive group over $A$, by~\cite[Lemma 5.3]{St-k1}
it is enough to prove the claim for every strictly isotropic simply connected group $G$ over $A$.
Assume that $G$ is simply connected.
Since $A(x)$ is a Dedekind domain, by Theorem~\ref{th:main} the kernel of $H^1_{\et}(A(x),G)\to H^1_{\et}(L(x),G)$
is trivial. It follows that
$$
\ker\left(H^1_{\et}(A[x],G)\to H^1_{\et}(L[x],G)\right)\subseteq
\ker\left(H^1_{\et}(A[x],G)\to H^1_{\et}(A(x),G)\right).
$$
Let $E$ be a principal $G$-bundle over $A[x]$ and let $f\in A[x]$ be a monic polynomial such that
$E_{A[x]_f}$ is trivial. Then $E$ extends to a principal $G$-bundle $\tilde E$ over $\mathbb{P}^1_A$ such that $\tilde E$ is trivial
at infinity. Then by~\cite[Theorem 1.7]{PSt2} $\tilde E|_{\mathbb{A}^1_A}$ is trivial. Hence $E$ is trivial over $A[x]$.
\end{proof}

\begin{lem}\label{lem:x2dim}
Let $A$ be a Dedekind domain and let $L$ be the fraction field of $A$.
Let $G$ be a strictly isotropic simply connected group over $A$.
Let $S\subseteq A[x]$ be a multiplicatively closed subset.
Then $H^1_{\et}(A[x]_S,G)\to H^1_{\et}(L(x),G)$ has trivial kernel.
\end{lem}
\begin{proof}
Since a $G$-torsor $E$ over $A[x]_S$ is finitely presented, it follows that any such torsor
is defined over $A[x]_f$ for a non-zero polynomial $f(x)\in A[x]$. Therefore, to prove the lemma, it is
enough to show that the map $H^1_{\et}(A[x]_f,G)\to H^1_{\et}(L(x),G)$ has trivial kernel.

Set $U=1+fA[x]\subseteq A[x]$.
Then $(A[x]_U)_f$ is a regular ring of dimension $\le 1$,
hence a Dedekind domain. Indeed, let $m$ be a maximal ideal of $A[x]$ not containing $f$. Since $(f,m)=A[x]$,
it follows that $m$ contains an element of $U$. Therefore, no maximal ideal of $A[x]$ survives in $A[x]_{Uf}$.

By Theorem~\ref{th:main} $H^1_{\et}((A[x]_U)_f,G)\to H^1_{\et}(L(x),G)$ has trivial kernel.
Let $E$ be a $G$-torsor defined over $A[x]_f$ and trivial over $L(x)$.
Hence there is $h(x)=1+f(x)g(x)\in U$ such that $E_h$ is trivial. Consider the commutative square
\begin{equation*}
\xymatrix@R=15pt@C=20pt{
A[x]\ar[r]\ar[d]&A[x]_{f}\ar[d]\\
A[x]_h\ar[r]&A[x]_{fh}\\
}
\end{equation*}
This is a patching diagram for torsors. The torsor $E$ is trivial over $A[x]_{fh}$. Then it can be extended
to a $G$-torsor $E'$ over $A[x]$ such that $E'_h$ is trivial. In particular, $E'$ is trivial over $L(x)$.
By~\cite[Prop. 2.2]{CTO} it follows that $E'_{L[x]}$ is trivial, and hence $(E'_{L[x]})|_{x=0}$ is trivial.
Then by Theorem~\ref{th:main} it follows that $E'|_{x=0}$ is trivial.
Then by by a generalization of Quillen's local-global principle~\cite[Theorem 3.2.5]{AHW}, to show that $E'$ is trivial over $A[x]$, it is enough to show that it is trivial over $A_m[x]$ for all
maximal ideals $m$ of $A$.
By Lemma~\ref{lem:A[x]} the kernel of $H^1_{\et}(A_m[x],G)\to H^1_{\et}(L[x],G)$ is trivial.
This proves the claim.
\end{proof}

\begin{proof}[Proof of Theorem~\ref{th:3dim}]
Let $\pi\in D$ be a prime element.
Let $E$ be a $G$-torsor over $R_f$.
Since $E$ trivializes over $L$,
there is $h\in R$ such that $E_{hf}$ is a trivial $G$-torsor. We can assume without loss of generality
that $h$ and $f$ have no common factors and $h\in m^2$, where $m$ is the maximal ideal of $R$.
Moreover, we can assume that $\pi$ does not divide $h$. Indeed,
if $\pi$ divides $f$, then it is a common factor of $f$ and $h$, so $\pi$ does not divide $f$.
Then by the Grothendieck-Serre
conjecture for discrete valuation rings there is $h\in R_f$ not contained in $(\pi)$ such that $E_h$ is trivial.

Since $D$ is of mixed characteristic, it is excellent. Then  by Dutta's version of Lindel's lemma~\cite[Theorem 1.3]{Dut}
there exists a regular local subring $(B,n)$ of $(R,m)$, with
$B/n=R/m$, and such that
\begin{enumerate}
\item
$B$ is a localization of a polynomial ring $W[x_1,x_2]$ at a maximal ideal
 $(\pi,\phi(x_1 ),x_2)$, where $x_2=f$, $\phi$ is a monic irreducible polynomial
in $W[x_1]$ and
$(W,(\pi))$ is a subring $R$, isomorphic to $D[y_1,\ldots,y_l]_{(\pi)}$ for some $l\ge 0$ (
in particular, $W$ is a dvr);
\item $R$ is a localization of an \'etale $B$-algebra;
\item There exists an element $b\in B\cap hR$ such that $B/bB\cong R/hR$ is an isomorphism. Furthermore $hR = bR$.
\end{enumerate}
Then $B_{x_2}/b\cong R_f/h$ and we have a patching square for $G$-torsors
\begin{equation*}
\xymatrix@R=15pt@C=20pt{
B_{x_2}\ar[r]\ar[d]&R_f\ar[d]\\
B_{x_2b}\ar[r]&R_{fh}\\
}
\end{equation*}
Consequently, there is a $G$-torsor over $B_{x_2}=\bigl(W[x_1,x_2]_{(\pi,\phi(x_1 ),x_2)}\bigr)_{x_2}$,
trivial over $B_{x_2b}$, that induces $E$. Denote this $G$-torsor
by the same letter $E$. It is enough to show that the $G$-torsor $E$ over $B_{x_2}$ is trivial.

Since the ideal $(\pi,\phi(x_1),x_2)\subseteq W[x_1,x_2]$ intersects $W[x_2]$ by the ideal $(\pi,x_2)$, the ring
$W[x_2]_{(\pi,x_2)}[x_1]$ is a subring of $B=W[x_1,x_2]_{(\pi,\phi(x_1 ),x_2)}$. Consequently, the ring
$X=(W[x_2]_{(\pi,x_2)})_{x_2}[x_1]$ is a subring of $B_{x_2}$. Since $W[x_1,x_2]\subseteq X\subseteq B_{x_2}$
and $B_{x_2}$ is a localization of $W[x_1,x_2]$, it follows that $B_{x_2}=X_S$ is also a localization of $X$
at a multiplicatively closed subset $S\subseteq X$.
Therefore, $B_{x_2}=X_S=A[x_1]_S$,
where $A=(W[x_2]_{(\pi,x_2)})_{x_2}$ is a Dedekind domain. Since the torsor $E$ over $B_{x_2}$ is trivial
over $B_{x_2b}$, we conclude by Lemma~\ref{lem:x2dim} that $E$ is trivial.
\end{proof}

\renewcommand{\refname}{References}

\end{document}